\documentclass{ametsocV5}

\usepackage{amsmath,amsfonts,amssymb,amsthm,mathtools,bm}
\usepackage{mathptmx}
\usepackage{newtxtext}
\usepackage{newtxmath}
\makeatletter
\AtBeginDocument{%
  \nolinenumbers
  \let\linenumbers\relax
  \let\internallinenumbers\relax

}
\makeatother
\theoremstyle{plain}
\newtheorem{theorem}{Theorem}[section]
\newtheorem{lemma}[theorem]{Lemma}
\newtheorem{proposition}[theorem]{Proposition}
\newtheorem{corollary}[theorem]{Corollary}
\theoremstyle{definition}
\newtheorem{definition}[theorem]{Definition}
\theoremstyle{remark}
\newtheorem{remark}[theorem]{Remark}

\newcommand{\dd}{\,\mathrm{d}}
\newcommand{\RR}{\mathbb R}

\title{Weighted Uniform Endpoint Majorants for Integrals Involving Modified Bessel Functions}

\authors{Yaoran Yang and Yutong Zhang\correspondingauthor{Yutong Zhang, yutongzhang@stu.scu.edu.cn}}
\affiliation{School of Mathematics, Sichuan University, 24 First Loop Road South Section I, Chengdu 610064, Sichuan, China}
\runningheadauthors{Yang and Zhang}

\abstract{We give an affirmative full-range solution to Gaunt's 2019 Open Problem~2.10. The problem asks whether, for every \(\nu>-1/2\) and \(0<\gamma<1\), the reciprocal-power integral \(\int_0^x e^{-\gamma t}I_\nu(t)t^{-\nu}\,\dd t\) is bounded by a constant multiple of \(e^{-\gamma x}I_{\nu+1}(x)x^{-\nu}\), uniformly for all \(x>0\). Earlier exponential-tilt estimates proved such endpoint majorants only under an additional smallness condition on \(\gamma\). We prove the estimate throughout the natural range \(0<\gamma<1\), with an explicit admissible constant. More generally, if \(\mu>-1\), \(q>-1\), \(0<\gamma<1\), and \(w(x)x^{-q}\) is nondecreasing on \((0,\infty)\), then for every \(\theta\in(\gamma,1)\), \(\int_0^x e^{-\gamma t}w(t)t^{-\mu}I_\mu(t)\,\dd t\) is controlled by an explicit multiple of \(e^{-\gamma x}w(x)x^{-\mu}I_{\mu+1}(x)\). The case \(w\equiv1\), \(q=0\), and \(\mu=\nu\) resolves Gaunt's problem. The weighted theorem also yields shifted-order and moment estimates, applies to approximate power weights and monotone regularly varying amplitudes, and provides lower endpoint estimates under a reversed comparison and two-sided estimates for exact power weights. We further analyze the sharp power-weighted quotient via endpoint expansions, a stationary equation, and parameter monotonicity.}

\begin{document}

\maketitle


\section{Introduction}

The modified Bessel function of the first kind is
\begin{equation}\label{Iseries}
 I_\alpha(x)=\sum_{k=0}^{\infty}\frac{(x/2)^{2k+\alpha}}{k!\Gamma(k+\alpha+1)},
 \qquad x>0.
\end{equation}
Finite integrals involving \(I_\alpha\) occur in approximation theory, probability and special-function estimates, but they usually do not reduce to simple endpoint expressions.  This has motivated a sequence of inequalities for modified Bessel integrals; see, among others, Gaunt \cite{Gaunt2014,Gaunt2018,Gaunt2021}.  Gaunt \cite{Gaunt2019} considered in particular the reciprocal-power family
\begin{equation}\label{GauntClass}
  \int_0^x \exp\!\left(-\gamma t\right)\frac{I_{\nu+n}(t)}{t^\nu}\,\dd t,
  \qquad x>0,
\end{equation}
and obtained sharp or asymptotically sharp endpoint estimates in several cases.  For the exponentially tilted upper bounds, however, the available result involved a restriction of the form \(0<\gamma<1/C_{\nu,n}\).  This leaves open the full natural tilt range \(0<\gamma<1\), which is exactly the range in which the integral has the same exponential growth rate as its endpoint majorant.

The specific question left by Gaunt is Open Problem~2.10 of \cite{Gaunt2019}: for \(\nu>-1/2\) and \(0<\gamma<1\), find a finite constant \(M_{\nu,\gamma}\) such that
\begin{equation}\label{GauntOpen}
 \int_0^x \exp\!\left(-\gamma t\right)\frac{I_\nu(t)}{t^\nu}\,\dd t
 <M_{\nu,\gamma}\exp\!\left(-\gamma x\right)\frac{I_{\nu+1}(x)}{x^\nu},
 \qquad x>0.
\end{equation}
The endpoint form in \eqref{GauntOpen} is forced.  Indeed,
\begin{equation}\label{GauntEndpointLimitsIntro}
 \frac{\displaystyle\int_0^x \exp\!\left(-\gamma t\right)I_\nu(t)t^{-\nu}\,\dd t}
 {\displaystyle\exp\!\left(-\gamma x\right)I_{\nu+1}(x)x^{-\nu}}
 \to 2(\nu+1)\quad(x\downarrow0),
 \qquad
 \to\frac1{1-\gamma}\quad(x\to\infty).
\end{equation}
Thus the difficulty is not local endpoint order.  The difficulty is global: the factor \(\exp\{\gamma(x-t)\}\), which appears after factoring out \(\exp(-\gamma x)\), changes the distribution of mass over \((0,x)\) and prevents the \(\gamma=0\) endpoint inequality from giving a uniform bound on the whole half-line.

The first result of this paper is that \eqref{GauntOpen} has a positive answer for the full range \(\nu>-1/2\), \(0<\gamma<1\).  More precisely, for every \(\theta\in(\gamma,1)\),
\begin{equation}\label{IntroGauntConstant}
 \int_0^x e^{-\gamma t}\frac{I_\nu(t)}{t^\nu}\,\dd t
 <
 \max\left\{
 2(\nu+1)\exp\!\left(\frac{2\gamma(\nu+2)\theta}{1-\theta}\right),
 \frac{\nu+1+\theta}{(\nu+2)\theta(\theta-\gamma)}
 \right\}
 e^{-\gamma x}\frac{I_{\nu+1}(x)}{x^\nu},
 \qquad x>0.
\end{equation}
The constants in \eqref{IntroGauntConstant} are not claimed to be sharp; their role is to prove finite uniform control explicitly throughout the whole tilt interval.

The main theorem is stronger than \eqref{IntroGauntConstant}.  Fix
\begin{equation}\label{MainRangeIntro}
 \mu>-1,
 \qquad q>-1,
 \qquad 0<\gamma<1,
\end{equation}
and write
\begin{equation}\label{YZintro}
 Y_{\mu,q}(x)=x^{q-\mu}I_\mu(x),
 \qquad
 Z_{\mu,q}(x)=x^{q-\mu}I_{\mu+1}(x).
\end{equation}
We prove a weighted endpoint inequality of the form
\begin{equation}\label{GeneralEndpointProblem}
 \int_0^x \exp\!\left(-\gamma t\right)w(t)t^{-\mu}I_\mu(t)\,\dd t
 \le M\exp\!\left(-\gamma x\right)w(x)x^{-\mu}I_{\mu+1}(x),
 \qquad x>0,
\end{equation}
under the one-sided power comparison
\begin{equation}\label{BasicWeightIntro}
 x^{-q}w(x)\quad\text{is nondecreasing on }(0,\infty).
\end{equation}
This condition is the natural sufficient structural assumption used below to compare lower endpoint values of the weight with the terminal value \(w(x)\).  For \(\theta\in(\gamma,1)\), define
\begin{equation}\label{Kintro}
 K_{\mu,q}=\max\left\{1,\frac{2(\mu+1)}{q+1}\right\}.
\end{equation}
When \(q\ge0\), Theorem~\ref{thm:main-weighted} gives \eqref{GeneralEndpointProblem} with
\begin{equation}\label{ClosedMIntro}
 M_{\mu,q,\gamma}(\theta)=
 \max\left\{
 K_{\mu,q}\exp\!\left(\frac{2\gamma(\mu+2)\theta}{1-\theta}\right),
 \frac{\mu+1+\theta}{(\mu+2)\theta(\theta-\gamma)}
 \right\}.
\end{equation}
For \(-1<q<0\), the same formulaic structure holds with a larger explicit threshold; see \eqref{Xdef} and \eqref{Mgeneral}.  The two terms in \eqref{ClosedMIntro} come from different parts of the interval: a small-endpoint estimate based on a sharp power inequality, and a tail estimate based on a lower bound for the Bessel ratio \(I_{\mu+1}/I_\mu\).

The main contributions can be summarized as follows.
\begin{enumerate}
\item We solve Gaunt's Open Problem~2.10 by proving \eqref{IntroGauntConstant}, hence obtaining a uniform endpoint majorant for \(I_\nu(t)t^{-\nu}\) for every \(\nu>-1/2\) and \(0<\gamma<1\).
\item We show that the open problem is the unweighted case of a stable weighted principle: if \(w(x)x^{-q}\) is nondecreasing, then the weighted integral in \eqref{GeneralEndpointProblem} is controlled by the same terminal value of the weight and the endpoint order \(I_{\mu+1}\).
\item We derive shifted-order and moment estimates.  Setting
\begin{equation}\label{ShiftIntro}
 \mu=\nu+n,
 \qquad q=a+n,
\end{equation}
turns the power-weighted inequality into
\begin{equation}\label{ShiftedIntro}
 \int_0^x \exp\!\left(-\gamma t\right)t^a\frac{I_{\nu+n}(t)}{t^\nu}\,\dd t
 \le M_{\nu+n,a+n,\gamma}(\theta)
 \exp\!\left(-\gamma x\right)x^a\frac{I_{\nu+n+1}(x)}{x^\nu},
 \qquad x>0,
\end{equation}
under the minimal convergence conditions \(\nu+n>-1\) and \(a+n>-1\).  For \(\nu>-1/2\), the choice \(a=n=0\) recovers \eqref{GauntOpen}.
\item We extend the endpoint theory beyond exact powers: approximate power weights, reversed one-sided comparisons, positive power mixtures and monotone regularly varying amplitudes are all treated without changing the endpoint scale.
\item We separate existence of a uniform majorant from sharp evaluation of the best constant.  For power weights we study the sharp quotient by endpoint expansions, a stationary equation, an optimized one-parameter constructive constant and parameter monotonicity.
\end{enumerate}

The relation with existing work is as follows.  The inequalities of Gaunt \cite{Gaunt2014,Gaunt2018,Gaunt2019,Gaunt2021} provide the immediate background, and the ratio estimates used here belong to the classical Bessel-ratio literature of Amos \cite{Amos1974}, Nasell \cite{Nasell1974,Nasell1978}, Soni \cite{Soni1965}, Jones \cite{Jones1968}, Baricz \cite{Baricz2010}, Baricz and Ponnusamy \cite{BariczPonnusamy2010}, Hornik and Grun \cite{HornikGrun2013}, Ruiz-Antolin and Segura \cite{RuizSegura2016}, and Segura \cite{Segura2011,Segura2023}.  A recent paper of Gaunt \cite{Gaunt2025} treats the different positive-power integral \(\int_0^x e^{-\gamma t}t^\nu I_\nu(t)\,\dd t\) and obtains refined constants for that specialized setting.  Some positive-power cases are contained in our weighted formulation, but our main point is different: the full-range resolution of the reciprocal-power endpoint problem \eqref{GauntOpen} and its extension to a general monotone-weight framework.

The paper is organized as follows.  Section \ref{sec:basic} records the ratio comparison and the sharp unweighted power inequality.  Section \ref{sec:admissible} proves the main weighted theorem and its approximate-weight variant.  Section \ref{sec:twosided} gives two-sided estimates.  Section \ref{sec:shifted} treats shifted order and shifted endpoint choices, including the obstruction for endpoints of order higher than \(I_{\mu+1}\).  Section \ref{sec:sharp} studies the sharp quotient and the optimized constructive constant.  Section \ref{sec:parameters} gives parameter comparison theorems.  Section \ref{sec:weights} treats positive power mixtures and monotone regularly varying weights.  Section \ref{sec:gaunt} states explicitly how the estimates recover and strengthen the corresponding Gaunt-type inequalities.

\section{Basic identities and the sharp unweighted power bound}\label{sec:basic}

We use standard recurrence and differentiation identities for modified Bessel functions; see Watson \cite{Watson1944} and the NIST Handbook \cite[Chapter 10]{NIST}. In particular,
\begin{equation}\label{BasicBesselIds}
 I_{\alpha-1}(x)-I_{\alpha+1}(x)=\frac{2\alpha}{x}I_\alpha(x),
 \qquad
 \frac{\dd}{\dd x}\{x^{-\alpha}I_\alpha(x)\}=x^{-\alpha}I_{\alpha+1}(x).
\end{equation}
For \(r_\alpha=I_{\alpha+1}/I_\alpha\) these imply
\begin{equation}\label{Riccati}
 r_\alpha'(x)=1-r_\alpha(x)^2-\frac{2\alpha+1}{x}r_\alpha(x).
\end{equation}
The normalized endpoint scales satisfy
\begin{equation}\label{LogDerivatives}
 \frac{\dd}{\dd x}\log Y_{\mu,q}(x)=r_\mu(x)+\frac qx,
 \qquad
 \frac{\dd}{\dd x}\log Z_{\mu,q}(x)=r_{\mu+1}(x)+\frac{q+1}{x}.
\end{equation}

\begin{lemma}\label{lemma:ratio}
If \(\alpha>-1\), then
\begin{equation}\label{RatioLower}
 r_\alpha(x)>\frac{x}{x+2\alpha+2},
 \qquad x>0.
\end{equation}
\end{lemma}

\begin{proof}
Let \(a=2\alpha+2>0\) and \(s(x)=x/(x+a)\). By \eqref{Riccati},
\[
 1-s(x)^2-\frac{a-1}{x}s(x)=\frac{(a+1)x+a}{(x+a)^2}>\frac{a}{(x+a)^2}=s'(x).
\]
At the origin,
\[
 r_\alpha(x)=\frac{x}{2\alpha+2}+O(x^3),
 \qquad
 s(x)=\frac{x}{2\alpha+2}-\frac{x^2}{(2\alpha+2)^2}+O(x^3),
\]
so \(r_\alpha>s\) on some interval \((0,\varepsilon)\). Suppose, for contradiction, that
\(r_\alpha-s\) vanishes somewhere on \((0,\infty)\), and let \(x_0\) be its first zero. Then
\(r_\alpha(x_0)=s(x_0)\) and \((r_\alpha-s)'(x_0)\le0\). But substituting
\(r_\alpha(x_0)=s(x_0)\) into \eqref{Riccati} and using the displayed strict differential inequality gives
\(r_\alpha'(x_0)-s'(x_0)>0\), a contradiction. Hence \eqref{RatioLower} holds for all \(x>0\).
\end{proof}

The following coefficient comparison is the elementary source of the sharp power constant appearing throughout the paper.

\begin{lemma}\label{lemma:power-unweighted}
Let \(\mu>-1\) and \(q>-1\). Then
\begin{equation}\label{PowerUnweightedUpper}
 \int_0^x t^{q-\mu}I_\mu(t)\,\dd t
 \le K_{\mu,q}x^{q-\mu}I_{\mu+1}(x),
 \qquad x>0,
\end{equation}
where
\begin{equation}\label{Kdef}
 K_{\mu,q}=\max\left\{1,\frac{2(\mu+1)}{q+1}\right\}.
\end{equation}
The constant \(K_{\mu,q}\) is best possible. Moreover,
\begin{equation}\label{PowerUnweightedLower}
 \int_0^x t^{q-\mu}I_\mu(t)\,\dd t
 \ge L_{\mu,q}x^{q-\mu}I_{\mu+1}(x),
 \qquad x>0,
\end{equation}
where
\begin{equation}\label{Ldef}
 L_{\mu,q}=\min\left\{1,\frac{2(\mu+1)}{q+1}\right\},
\end{equation}
and \(L_{\mu,q}\) is also best possible. Equality holds in both estimates for all \(x>0\) when \(q=2\mu+1\).
\end{lemma}

\begin{proof}
From \eqref{Iseries},
\[
 \int_0^x t^{q-\mu}I_\mu(t)\,\dd t
 =\sum_{k=0}^{\infty}
 \frac{x^{2k+q+1}}{2^{2k+\mu}k!\Gamma(k+\mu+1)(2k+q+1)},
\]
while
\[
 x^{q-\mu}I_{\mu+1}(x)
 =\sum_{k=0}^{\infty}
 \frac{x^{2k+q+1}}{2^{2k+\mu+1}k!\Gamma(k+\mu+2)}.
\]
The ratio of the corresponding coefficients is
\begin{equation}\label{CoefficientRatio}
 A_k=\frac{2(k+\mu+1)}{2k+q+1}.
\end{equation}
The sequence \(A_k\) is monotone, since the sign of \(A_{k+1}-A_k\) is the sign of \(q-2\mu-1\), and \(A_k\to1\). Hence all coefficient ratios lie between \(L_{\mu,q}\) and \(K_{\mu,q}\). This proves \eqref{PowerUnweightedUpper} and \eqref{PowerUnweightedLower}. Sharpness follows from the first coefficient if \(q<2\mu+1\), from the limiting coefficient if \(q>2\mu+1\), and from coefficient equality if \(q=2\mu+1\).
\end{proof}

\begin{remark}\label{remark:classical}
For \(q=0\) and \(\mu>-1/2\), Lemma \ref{lemma:power-unweighted} gives
\[
 \int_0^x \frac{I_\mu(t)}{t^\mu}\,\dd t
 \le 2(\mu+1)\frac{I_{\mu+1}(x)}{x^\mu},
\]
which is the unweighted endpoint inequality underlying Gaunt's open problem. The coefficient proof shows at the same time why the constant is sharp at the origin.
\end{remark}

\section{Admissible weights and the main endpoint theorem}\label{sec:admissible}

\subsection{Exact one-sided power weights}

\begin{definition}\label{def:admissible}
Let \(q>-1\). A positive, finite and measurable function \(w\) on \((0,\infty)\) is called an upper \(q\)-power weight if
\begin{equation}\label{UpperPowerWeight}
 v_q(x):=\frac{w(x)}{x^q}
 \quad\text{is nondecreasing on }(0,\infty).
\end{equation}
Equivalently, for \(0<t\le x\),
\begin{equation}\label{UpperPowerRatio}
 \frac{w(t)}{w(x)}\le\left(\frac tx\right)^q.
\end{equation}
If \(w\) is absolutely continuous, \eqref{UpperPowerWeight} is implied by
\begin{equation}\label{UpperDifferential}
 \frac{w'(x)}{w(x)}\ge\frac qx
 \quad\text{for a.e. }x>0.
\end{equation}
For every fixed upper endpoint \(x\), condition \eqref{UpperPowerRatio} gives
\(w(t)\le w(x)x^{-q}t^q\) on \((0,x]\); hence the integrals below are finite because \(q>-1\) and \(t^{-\mu}I_\mu(t)=O(1)\) as \(t\downarrow0\).
\end{definition}

For \(\theta\in(0,1)\) define the threshold
\begin{equation}\label{Xdef}
 X_{\mu,q}(\theta)=
 \begin{cases}
 \displaystyle \frac{2(\mu+2)\theta}{1-\theta}, & q\ge0,\\[1.2em]
 \displaystyle
 \max\left\{
 \frac{2(\mu+2)\theta}{1-\theta},
 \frac{(2\mu+2)\theta-q+
 \sqrt{\{(2\mu+2)\theta-q\}^2-4(1-\theta)q(2\mu+2)}}
 {2(1-\theta)}
 \right\},& -1<q<0.
 \end{cases}
\end{equation}
Also put
\begin{equation}\label{betadef}
 \beta_{\mu,q}(\theta)=
 \frac{X_{\mu,q}(\theta)}{X_{\mu,q}(\theta)+2\mu+2}.
\end{equation}
When \(q\ge0\) this reduces to
\begin{equation}\label{BetaPositive}
 \beta_{\mu,q}(\theta)=\frac{(\mu+2)\theta}{\mu+1+\theta}.
\end{equation}

\begin{lemma}\label{lemma:growth}
Let \(\mu>-1\), \(q>-1\) and \(0<\theta<1\). If \(x\ge X_{\mu,q}(\theta)\), then
\begin{equation}\label{GrowthIneqs}
 r_\mu(x)+\frac qx\ge\theta,
 \qquad
 r_{\mu+1}(x)+\frac{q+1}{x}\ge\theta,
 \qquad
 r_\mu(x)\ge\beta_{\mu,q}(\theta).
\end{equation}
Consequently, for \(X_{\mu,q}(\theta)\le t\le x\),
\begin{equation}\label{GrowthY}
 Y_{\mu,q}(t)\le \exp\!\left(-\theta(x-t)\right)Y_{\mu,q}(x),
 \qquad
 Z_{\mu,q}(t)\le \exp\!\left(-\theta(x-t)\right)Z_{\mu,q}(x).
\end{equation}
\end{lemma}

\begin{proof}
The bound on \(r_\mu\) follows immediately from Lemma \ref{lemma:ratio}. Since \(q+1>0\),
\[
 r_{\mu+1}(x)+\frac{q+1}{x}\ge r_{\mu+1}(x)>\frac{x}{x+2\mu+4}\ge\theta
\]
whenever \(x\ge2(\mu+2)\theta/(1-\theta)\). If \(q\ge0\), then
\[
 r_\mu(x)+\frac qx\ge r_\mu(x)>\frac{x}{x+2\mu+2}\ge\theta
\]
under the same threshold. If \(-1<q<0\), define
\[
 h(x)=\frac{x}{x+2\mu+2}+\frac qx.
\]
Then \(h'(x)=(2\mu+2)/(x+2\mu+2)^2-q/x^2>0\), so \(h\) is strictly increasing. The second entry in \eqref{Xdef} is the unique positive root of \(h(x)=\theta\), obtained from the quadratic
\[
 (1-\theta)x^2+\{q-(2\mu+2)\theta\}x+q(2\mu+2)=0.
\]
Therefore \(h(x)\ge\theta\) for \(x\ge X_{\mu,q}(\theta)\), proving the first inequality. The exponential comparisons \eqref{GrowthY} follow by integrating the logarithmic derivatives in \eqref{LogDerivatives}.
\end{proof}

\begin{theorem}\label{thm:main-weighted}
Let \(\mu>-1\), \(q>-1\), \(0<\gamma<1\), and let \(w\) be an upper \(q\)-power weight. For every \(\theta\in(\gamma,1)\),
\begin{equation}\label{MainWeightedBound}
 \int_0^x \exp\!\left(-\gamma t\right)w(t)t^{-\mu}I_\mu(t)\,\dd t
 \le M_{\mu,q,\gamma}(\theta)
 \exp\!\left(-\gamma x\right)w(x)x^{-\mu}I_{\mu+1}(x),
 \qquad x>0,
\end{equation}
where
\begin{equation}\label{Mgeneral}
 M_{\mu,q,\gamma}(\theta)=
 \max\left\{
 K_{\mu,q}\exp\!\left(\gamma X_{\mu,q}(\theta)\right),
 \frac{1}{\beta_{\mu,q}(\theta)(\theta-\gamma)}
 \right\}.
\end{equation}
\end{theorem}

\begin{proof}
Set \(X=X_{\mu,q}(\theta)\), \(\beta=\beta_{\mu,q}(\theta)\), \(\delta=\theta-\gamma\), and
\[
 J(x)=\exp\!\left(\gamma x\right)\int_0^x \exp\!\left(-\gamma t\right)w(t)t^{-\mu}I_\mu(t)\,\dd t.
\]
Let
\[
 A=K_{\mu,q}\exp\!\left(\gamma X\right),
 \qquad
 C=\frac1{\beta\delta}.
\]
If \(0<x\le X\), then \eqref{UpperPowerRatio} and Lemma \ref{lemma:power-unweighted} give
\[
\begin{aligned}
 J(x)
 &\le \exp\!\left(\gamma X\right)\int_0^x w(t)t^{-\mu}I_\mu(t)\,\dd t \\
 &\le \exp\!\left(\gamma X\right)w(x)x^{-q}
       \int_0^x t^{q-\mu}I_\mu(t)\,\dd t \\
 &\le A\,w(x)x^{-\mu}I_{\mu+1}(x).
\end{aligned}
\]
Now assume \(x>X\) and split \(J=J_0+J_1\), where \(J_0\) is the contribution of \((0,X)\) and \(J_1\) the contribution of \((X,x)\). The previous estimate at \(X\) gives
\begin{equation}\label{JXbound}
 J(X)\le A w(X)X^{-\mu}I_{\mu+1}(X).
\end{equation}
For \(X\le t\le x\), \eqref{UpperPowerWeight} and Lemma \ref{lemma:growth} imply
\[
 w(t)t^{-\mu}I_\mu(t)\le \exp\!\left(-\theta(x-t)\right)w(x)x^{-\mu}I_\mu(x),
\]
and
\[
 w(t)t^{-\mu}I_{\mu+1}(t)\le \exp\!\left(-\theta(x-t)\right)w(x)x^{-\mu}I_{\mu+1}(x)
\]
whenever \(X\le t\le x\). Thus
\[
 J_0(x)\le A\exp\!\left(-\delta(x-X)\right)w(x)x^{-\mu}I_{\mu+1}(x).
\]
Also,
\[
\begin{aligned}
 J_1(x)
 &\le w(x)x^{-\mu}I_\mu(x)
    \int_X^x \exp\!\left(-\delta(x-t)\right)\,\dd t\\
 &\le \frac{1-\exp\!\left(-\delta(x-X)\right)}{\delta r_\mu(x)}
    w(x)x^{-\mu}I_{\mu+1}(x)\\
 &\le C\{1-\exp\!\left(-\delta(x-X)\right)\}
    w(x)x^{-\mu}I_{\mu+1}(x).
\end{aligned}
\]
Therefore
\[
 J(x)\le
 \{A\exp\!\left(-\delta(x-X)\right)+C(1-\exp\!\left(-\delta(x-X)\right))\}
 w(x)x^{-\mu}I_{\mu+1}(x),
\]
and the coefficient is at most \(\max\{A,C\}\). Multiplying by \(\exp\!\left(-\gamma x\right)\) proves the theorem.
\end{proof}

\begin{corollary}\label{cor:closed-q-positive}
Let \(\mu>-1\), \(q\ge0\), \(0<\gamma<1\), and let \(w\) be an upper \(q\)-power weight. Then \eqref{MainWeightedBound} holds with
\begin{equation}\label{Mpositive}
 M_{\mu,q,\gamma}(\theta)=
 \max\left\{
 K_{\mu,q}\exp\!\left(\frac{2\gamma(\mu+2)\theta}{1-\theta}\right),
 \frac{\mu+1+\theta}{(\mu+2)\theta(\theta-\gamma)}
 \right\},
 \qquad \gamma<\theta<1.
\end{equation}
In particular, the choice \(\theta=(1+\gamma)/2\) gives
\begin{equation}\label{Mclosed}
 M_{\mu,q,\gamma}^{\rm cl}=
 \max\left\{
 K_{\mu,q}\exp\!\left(\frac{2\gamma(\mu+2)(1+\gamma)}{1-\gamma}\right),
 \frac{2(2\mu+3+\gamma)}{(\mu+2)(1-\gamma^2)}
 \right\}.
\end{equation}
\end{corollary}

\begin{proof}
Substitute \eqref{Xdef} and \eqref{BetaPositive} into \eqref{Mgeneral}.
\end{proof}

\subsection{Approximate power weights}

The monotonicity assumption \eqref{UpperPowerWeight} can be weakened. The exponential tilt can absorb a controlled failure of monotonicity.

\begin{definition}\label{def:approx}
Let \(q>-1\) and \(\eta\ge0\). A positive, finite and measurable function \(w\) belongs to the approximate upper class \(\mathcal A_q(\eta)\) if
\begin{equation}\label{ApproxClass}
 \frac{w(t)}{t^q}\le \exp\!\left(\eta(x-t)\right)\frac{w(x)}{x^q},
 \qquad 0<t\le x<\infty.
\end{equation}
If \(w\) is absolutely continuous, \eqref{ApproxClass} is implied by
\begin{equation}\label{ApproxDerivative}
 \frac{w'(x)}{w(x)}\ge \frac qx-\eta
 \quad\text{for a.e. }x>0.
\end{equation}
More generally, if
\begin{equation}\label{RhoDerivative}
 \frac{w'(x)}{w(x)}\ge \frac qx-\rho(x),
\end{equation}
where \(\rho\ge0\) is locally integrable and
\begin{equation}\label{RhoAverage}
 \int_t^x \rho(s)\,\dd s\le\eta(x-t),
 \qquad 0<t\le x,
\end{equation}
then \(w\in\mathcal A_q(\eta)\).
\end{definition}

\begin{theorem}\label{thm:approx}
Let \(\mu>-1\), \(q>-1\), \(0<\gamma<1\), and let \(w\in\mathcal A_q(\eta)\) with \(0\le\eta<1-\gamma\). For every
\begin{equation}\label{ThetaApproxRange}
 \gamma+\eta<\theta<1,
\end{equation}
one has
\begin{equation}\label{ApproxBound}
 \int_0^x \exp\!\left(-\gamma t\right)w(t)t^{-\mu}I_\mu(t)\,\dd t
 \le M_{\mu,q,\gamma}^{(\eta)}(\theta)
 \exp\!\left(-\gamma x\right)w(x)x^{-\mu}I_{\mu+1}(x),
 \qquad x>0,
\end{equation}
where
\begin{equation}\label{Mapprox}
 M_{\mu,q,\gamma}^{(\eta)}(\theta)=
 \max\left\{
 K_{\mu,q}\exp\!\left((\gamma+\eta)X_{\mu,q}(\theta)\right),
 \frac{1}{\beta_{\mu,q}(\theta)(\theta-\gamma-\eta)}
 \right\}.
\end{equation}
\end{theorem}

\begin{proof}
Set \(X=X_{\mu,q}(\theta)\), \(\beta=\beta_{\mu,q}(\theta)\), and
\(\delta=\theta-\gamma-\eta>0\). Put
\[
 J(x)=\exp\!\left(\gamma x\right)\int_0^x\exp\!\left(-\gamma t\right)w(t)t^{-\mu}I_\mu(t)\,\dd t.
\]
Let
\[
 A=K_{\mu,q}\exp\!\left((\gamma+\eta)X\right),
 \qquad
 C=\frac1{\beta\delta}.
\]
If \(0<x\le X\), then \eqref{ApproxClass} gives
\[
 w(t)\le \exp\!\left(\eta(x-t)\right)w(x)x^{-q}t^q,
 \qquad 0<t\le x.
\]
Hence, using \(\exp\{(\gamma+\eta)(x-t)\}\le \exp\{(\gamma+\eta)X\}\) and Lemma \ref{lemma:power-unweighted},
\[
\begin{aligned}
 J(x)
 &\le w(x)x^{-q}\int_0^x
     \exp\!\left((\gamma+\eta)(x-t)\right)t^{q-\mu}I_\mu(t)\,\dd t \\
 &\le \exp\!\left((\gamma+\eta)X\right)w(x)x^{-q}
     \int_0^x t^{q-\mu}I_\mu(t)\,\dd t \\
 &\le A\,w(x)x^{-\mu}I_{\mu+1}(x).
\end{aligned}
\]
Now let \(x>X\) and split \(J=J_0+J_1\), where \(J_0\) is the contribution of \((0,X)\) and \(J_1\) the contribution of \((X,x)\). The preceding estimate at \(X\) gives
\[
 J(X)\le A w(X)X^{-\mu}I_{\mu+1}(X).
\]
For \(X\le t\le x\), condition \eqref{ApproxClass} and Lemma \ref{lemma:growth} imply
\[
 w(t)t^{-\mu}I_\mu(t)
 \le \exp\!\left(-\theta(x-t)+\eta(x-t)\right)w(x)x^{-\mu}I_\mu(x),
\]
and similarly
\[
 w(t)t^{-\mu}I_{\mu+1}(t)
 \le \exp\!\left(-\theta(x-t)+\eta(x-t)\right)w(x)x^{-\mu}I_{\mu+1}(x).
\]
Consequently,
\[
 J_0(x)=\exp\!\left(\gamma(x-X)\right)J(X)
 \le A\exp\!\left(-\delta(x-X)\right)w(x)x^{-\mu}I_{\mu+1}(x),
\]
and
\[
\begin{aligned}
 J_1(x)
 &\le w(x)x^{-\mu}I_\mu(x)
    \int_X^x \exp\!\left(-\delta(x-t)\right)\,\dd t\\
 &\le \frac{1-\exp\!\left(-\delta(x-X)\right)}{\delta r_\mu(x)}
    w(x)x^{-\mu}I_{\mu+1}(x)\\
 &\le C\{1-\exp\!\left(-\delta(x-X)\right)\}
    w(x)x^{-\mu}I_{\mu+1}(x).
\end{aligned}
\]
Thus
\[
 J(x)\le
 \{A\exp\!\left(-\delta(x-X)\right)+C(1-\exp\!\left(-\delta(x-X)\right))\}
 w(x)x^{-\mu}I_{\mu+1}(x),
\]
and the coefficient is at most \(\max\{A,C\}\). Multiplication by \(\exp(-\gamma x)\) proves \eqref{ApproxBound}.
\end{proof}

\begin{remark}\label{remark:exp-defect}
Theorem \ref{thm:approx} is useful for weights of the form
\[
   w(x)=x^q\exp\!\left(-\eta x\right)L(x),
\]
where \(L\) is nondecreasing. Such a weight need not satisfy \eqref{UpperPowerWeight} when \(\eta>0\); for instance, this fails when \(L\) is constant. It nevertheless belongs to \(\mathcal A_q(\eta)\). The price paid is exactly the replacement of the tilt \(\gamma\) by the effective tilt \(\gamma+\eta\).
\end{remark}

\section{Two-sided weighted endpoint estimates}\label{sec:twosided}

The upper theory is one-sided because an upper endpoint majorant requires lower endpoint values of the weight to be controlled by the value at \(x\). Reversing the monotonicity gives lower endpoint estimates.

\begin{definition}\label{def:lowerweight}
A positive, finite and measurable weight \(w\) is called a lower \(q\)-power weight if
\begin{equation}\label{LowerPowerWeight}
 \frac{w(x)}{x^q}\quad\text{is nonincreasing on }(0,\infty).
\end{equation}
Equivalently,
\begin{equation}\label{LowerPowerRatio}
 \frac{w(t)}{w(x)}\ge\left(\frac tx\right)^q,
 \qquad 0<t\le x.
\end{equation}
\end{definition}

\begin{theorem}\label{thm:lower}
Let \(\mu>-1\), \(q>-1\), \(0\le\gamma<1\), and let \(w\) be a lower \(q\)-power weight. Then
\begin{equation}\label{LowerWeighted}
 \int_0^x \exp\!\left(-\gamma t\right)w(t)t^{-\mu}I_\mu(t)\,\dd t
 \ge L_{\mu,q}\exp\!\left(-\gamma x\right)w(x)x^{-\mu}I_{\mu+1}(x),
 \qquad x>0,
\end{equation}
where \(L_{\mu,q}\) is defined in \eqref{Ldef}. The constant is sharp when \(\gamma=0\) and \(w(x)=x^q\).
\end{theorem}

\begin{proof}
Using \eqref{LowerPowerRatio} and \(\exp\!\left(-\gamma t\right)\ge\exp\!\left(-\gamma x\right)\),
\[
\begin{aligned}
 \int_0^x \exp\!\left(-\gamma t\right)w(t)t^{-\mu}I_\mu(t)\,\dd t
 &\ge \exp\!\left(-\gamma x\right)w(x)x^{-q}
     \int_0^x t^{q-\mu}I_\mu(t)\,\dd t  \\
 &\ge L_{\mu,q}\exp\!\left(-\gamma x\right)w(x)x^{-\mu}I_{\mu+1}(x),
\end{aligned}
\]
by Lemma \ref{lemma:power-unweighted}.
\end{proof}

For exact power weights both inequalities are available simultaneously.

\begin{corollary}\label{cor:twosided-power}
Let \(\mu>-1\), \(q>-1\), \(0<\gamma<1\), and \(\gamma<\theta<1\). Then
\begin{equation}\label{TwoSidedPower}
 L_{\mu,q}
 \le
 \frac{\displaystyle\int_0^x \exp\!\left(-\gamma t\right)t^{q-\mu}I_\mu(t)\,\dd t}
 {\displaystyle\exp\!\left(-\gamma x\right)x^{q-\mu}I_{\mu+1}(x)}
 \le
 M_{\mu,q,\gamma}(\theta),
 \qquad x>0.
\end{equation}
At \(\gamma=0\) the best possible upper and lower constants in \eqref{TwoSidedPower} are respectively \(K_{\mu,q}\) and \(L_{\mu,q}\).
\end{corollary}

\begin{proof}
The lower bound is Theorem \ref{thm:lower}; the upper bound is Theorem \ref{thm:main-weighted} with \(w(x)=x^q\). The sharpness at \(\gamma=0\) is contained in Lemma \ref{lemma:power-unweighted}.
\end{proof}

\begin{remark}\label{remark:twosided}
For \(\gamma>0\) the lower constant in \eqref{TwoSidedPower} is not generally sharp. The endpoint limit at zero equals \(2(\mu+1)/(q+1)\), whereas the endpoint limit at infinity equals \(1/(1-\gamma)\). The lower estimate is nevertheless global, explicit and of the correct endpoint scale.
\end{remark}

\section{Shifted order and admissible endpoint choices}\label{sec:shifted}

\subsection{The principal shifted endpoint}

Theorem \ref{thm:main-weighted} gives a systematic shifted-order family.

\begin{theorem}\label{thm:shifted}
Let \(\nu,n,a\in\RR\) satisfy
\begin{equation}\label{ShiftRange}
 \nu+n>-1,\qquad a+n>-1.
\end{equation}
Let \(0<\gamma<1\) and \(\gamma<\theta<1\). Then
\begin{equation}\label{ShiftedBound}
 \int_0^x \exp\!\left(-\gamma t\right)t^a\frac{I_{\nu+n}(t)}{t^\nu}\,\dd t
 \le M_{\nu+n,a+n,\gamma}(\theta)
 \exp\!\left(-\gamma x\right)x^a\frac{I_{\nu+n+1}(x)}{x^\nu},
 \qquad x>0.
\end{equation}
If \(a+n\ge0\), then
\begin{equation}\label{ShiftedClosedConstant}
 M_{\nu+n,a+n,\gamma}(\theta)=
 \max\left\{
 K_{\nu+n,a+n}\exp\!\left(\frac{2\gamma(\nu+n+2)\theta}{1-\theta}\right),
 \frac{\nu+n+1+\theta}{(\nu+n+2)\theta(\theta-\gamma)}
 \right\},
\end{equation}
where
\begin{equation}\label{ShiftedK}
 K_{\nu+n,a+n}=\max\left\{1,\frac{2(\nu+n+1)}{a+n+1}\right\}.
\end{equation}
\end{theorem}

\begin{proof}
Set \(\mu=\nu+n\) and \(q=a+n\) in Theorem \ref{thm:main-weighted} with \(w(x)=x^q\). Then
\[
 t^{q-\mu}I_\mu(t)=t^{a+n-(\nu+n)}I_{\nu+n}(t)=t^a\frac{I_{\nu+n}(t)}{t^\nu},
\]
and the endpoint has the asserted form.
\end{proof}

\begin{corollary}\label{cor:moment}
Let \(\nu>-1\), \(m>-1\), \(0<\gamma<1\) and \(\gamma<\theta<1\). Then
\begin{equation}\label{MomentBound}
 \int_0^x \exp\!\left(-\gamma t\right)t^m\frac{I_\nu(t)}{t^\nu}\,\dd t
 \le M_{\nu,m,\gamma}(\theta)
 \exp\!\left(-\gamma x\right)x^m\frac{I_{\nu+1}(x)}{x^\nu},
 \qquad x>0.
\end{equation}
For \(m\ge0\) the constant is given by \eqref{Mpositive} with \((\mu,q)=(\nu,m)\).
\end{corollary}

\subsection{Other endpoint orders}

It is natural to ask whether \(I_{\mu+1}\) in the endpoint may be replaced by \(I_{\mu+\kappa}\). The next result shows that orders above \(\mu+1\) are impossible with the same power of \(x\), whereas lower endpoint orders are admissible under the standard order-monotonicity range.

\begin{lemma}\label{lemma:order}
For each \(x>0\), the map \(\alpha\mapsto I_\alpha(x)\) is decreasing on \([-1/2,\infty)\).
\end{lemma}

\begin{proof}
This is a classical Soni-type monotonicity theorem for modified Bessel functions; see Soni \cite{Soni1965}, Jones \cite{Jones1968}, and the discussion in Nasell \cite{Nasell1974}. We use it only in this standard form.
\end{proof}

\begin{theorem}\label{thm:endpoint-order}
Let \(\mu>-1\), \(q>-1\), \(0<\gamma<1\), and let \(\kappa\in\RR\).

\noindent\textnormal{(i)} If \(\kappa>1\) and \(\mu+\kappa>-1\), then there is no finite constant \(C\) such that
\begin{equation}\label{ImpossibleEndpoint}
 \int_0^x \exp\!\left(-\gamma t\right)t^{q-\mu}I_\mu(t)\,\dd t
 \le C\exp\!\left(-\gamma x\right)x^{q-\mu}I_{\mu+\kappa}(x)
 \qquad\text{for all }x>0.
\end{equation}

\noindent\textnormal{(ii)} If
\begin{equation}\label{EndpointOrderRange}
 -\frac12\le \mu+\kappa\le \mu+1,
\end{equation}
then for every \(\theta\in(\gamma,1)\),
\begin{equation}\label{LowerOrderEndpoint}
 \int_0^x \exp\!\left(-\gamma t\right)t^{q-\mu}I_\mu(t)\,\dd t
 \le M_{\mu,q,\gamma}(\theta)
 \exp\!\left(-\gamma x\right)x^{q-\mu}I_{\mu+\kappa}(x),
 \qquad x>0.
\end{equation}
\end{theorem}

\begin{proof}
For (i), the first term of \eqref{Iseries} gives
\[
 \int_0^x \exp\!\left(-\gamma t\right)t^{q-\mu}I_\mu(t)\,\dd t
 \sim \frac{x^{q+1}}{2^\mu(q+1)\Gamma(\mu+1)},
\]
whereas
\[
 \exp\!\left(-\gamma x\right)x^{q-\mu}I_{\mu+\kappa}(x)
 \sim \frac{x^{q+\kappa}}{2^{\mu+\kappa}\Gamma(\mu+\kappa+1)}.
\]
The quotient is asymptotic to a positive multiple of \(x^{1-\kappa}\), which diverges as \(x\downarrow0\) when \(\kappa>1\).

For (ii), Theorem \ref{thm:main-weighted} with \(w(x)=x^q\) gives the bound with \(I_{\mu+1}(x)\). By Lemma \ref{lemma:order}, condition \eqref{EndpointOrderRange} implies \(I_{\mu+\kappa}(x)\ge I_{\mu+1}(x)\). Hence the same constant works for the larger endpoint.
\end{proof}

In shifted notation this gives the following endpoint flexibility.

\begin{corollary}\label{cor:shifted-endpoint-order}
Let \(\nu+n>-1\), \(a+n>-1\), \(0<\gamma<1\), and \(\gamma<\theta<1\). If
\begin{equation}\label{ShiftedEndpointRange}
 -\frac12\le \nu+n+k\le \nu+n+1,
\end{equation}
then
\begin{equation}\label{ShiftedEndpointOrderBound}
 \int_0^x \exp\!\left(-\gamma t\right)t^a\frac{I_{\nu+n}(t)}{t^\nu}\,\dd t
 \le M_{\nu+n,a+n,\gamma}(\theta)
 \exp\!\left(-\gamma x\right)x^a\frac{I_{\nu+n+k}(x)}{x^\nu},
 \qquad x>0.
\end{equation}
If \(k>1\) and \(\nu+n+k>-1\), no finite bound of this form can hold for all \(x>0\).
\end{corollary}

\section{Sharp power-weighted quotients and optimized constants}\label{sec:sharp}

For power weights define
\begin{equation}\label{Rdef}
 R_{\mu,q,\gamma}(x)=
 \frac{\displaystyle\int_0^x \exp\!\left(-\gamma t\right)t^{q-\mu}I_\mu(t)\,\dd t}
 {\displaystyle\exp\!\left(-\gamma x\right)x^{q-\mu}I_{\mu+1}(x)},
 \qquad x>0.
\end{equation}
The sharp endpoint constant is
\begin{equation}\label{MstarDef}
 M^*_{\mu,q,\gamma}=\sup_{x>0}R_{\mu,q,\gamma}(x).
\end{equation}
Theorem \ref{thm:main-weighted} implies \(M^*_{\mu,q,\gamma}<\infty\) for every \(\mu>-1\), \(q>-1\) and \(0<\gamma<1\).

\begin{proposition}\label{prop:limits}
Let \(\mu>-1\), \(q>-1\) and \(0<\gamma<1\). Then
\begin{equation}\label{Rlimit0}
 \lim_{x\downarrow0}R_{\mu,q,\gamma}(x)=\frac{2(\mu+1)}{q+1},
\end{equation}
and
\begin{equation}\label{RlimitInf}
 \lim_{x\to\infty}R_{\mu,q,\gamma}(x)=\frac1{1-\gamma}.
\end{equation}
Consequently,
\begin{equation}\label{MstarLower}
 M^*_{\mu,q,\gamma}\ge
 \max\left\{\frac{2(\mu+1)}{q+1},\frac1{1-\gamma}\right\}.
\end{equation}
\end{proposition}

\begin{proof}
The first limit follows from \eqref{Iseries}. For the second, put \(\lambda=1-\gamma\) and \(s=q-\mu-1/2\). The standard asymptotic formula
\[
 I_\alpha(x)=\frac{\exp(x)}{\sqrt{2\pi x}}\{1+O(x^{-1})\},\qquad x\to\infty,
\]
gives
\[
 t^{q-\mu}I_\mu(t)=\frac{\exp(t) t^s}{\sqrt{2\pi}}\{1+O(t^{-1})\}.
\]
Endpoint integration by parts yields
\[
 \int_0^x \exp\!\left(-\gamma t\right)t^{q-\mu}I_\mu(t)\,\dd t
 \sim \frac{\exp\!\left(\lambda x\right)x^s}{\sqrt{2\pi}\lambda},
\]
while the denominator in \eqref{Rdef} is asymptotic to \(\exp\!\left(\lambda x\right)x^s/\sqrt{2\pi}\). This proves \eqref{RlimitInf}.
\end{proof}

\begin{proposition}\label{prop:expansions}
Let \(\mu>-1\), \(q>-1\) and \(0<\gamma<1\). As \(x\downarrow0\),
\begin{align}\label{SmallExpansion}
 R_{\mu,q,\gamma}(x)
 ={}&\frac{2(\mu+1)}{q+1}
 \Bigg[1+\frac{\gamma x}{q+2} \\
 &+\left\{
 \frac{\gamma^2}{(q+2)(q+3)}
 +\frac{q+1}{4(\mu+1)(q+3)}
 -\frac1{4(\mu+2)}
 \right\}x^2+O(x^3)\Bigg].
\end{align}
As \(x\to\infty\),
\begin{equation}\label{LargeExpansion}
 R_{\mu,q,\gamma}(x)=
 \frac1{1-\gamma}+
 \frac{(\mu+1/2)(2-\gamma)-q}{(1-\gamma)^2x}
 +O(x^{-2}).
\end{equation}
\end{proposition}

\begin{proof}
At zero,
\[
 t^{q-\mu}I_\mu(t)=\frac{t^q}{2^\mu\Gamma(\mu+1)}
 \left(1+\frac{t^2}{4(\mu+1)}+O(t^4)\right),
\]
and
\[
 x^{q-\mu}I_{\mu+1}(x)=\frac{x^{q+1}}{2^{\mu+1}\Gamma(\mu+2)}
 \left(1+\frac{x^2}{4(\mu+2)}+O(x^4)\right).
\]
Writing \(t=xu\) and expanding \(\exp\!\left(\gamma(x-t)\right)\) gives \eqref{SmallExpansion}.

At infinity, let \(\lambda=1-\gamma\), \(s=q-\mu-1/2\) and \(c_\alpha=-(4\alpha^2-1)/8\). Then
\[
 t^{q-\mu}I_\mu(t)=\frac{\exp(t) t^s}{\sqrt{2\pi}}
 \left(1+\frac{c_\mu}{t}+O(t^{-2})\right).
\]
An endpoint integration by parts gives
\[
 \int_0^x \exp\!\left(-\gamma t\right)t^{q-\mu}I_\mu(t)\,\dd t
 =\frac{\exp\!\left(\lambda x\right)x^s}{\sqrt{2\pi}\lambda}
 \left(1+\frac{c_\mu-s/\lambda}{x}+O(x^{-2})\right).
\]
The endpoint denominator is
\[
 \exp\!\left(-\gamma x\right)x^{q-\mu}I_{\mu+1}(x)
 =\frac{\exp\!\left(\lambda x\right)x^s}{\sqrt{2\pi}}
 \left(1+\frac{c_{\mu+1}}{x}+O(x^{-2})\right).
\]
Since \(c_\mu-c_{\mu+1}=\mu+1/2\), \eqref{LargeExpansion} follows.
\end{proof}

\begin{corollary}\label{cor:strict-small}
For \(\mu>-1\), \(q>-1\) and \(0<\gamma<1\),
\begin{equation}\label{StrictSmall}
 M^*_{\mu,q,\gamma}>\frac{2(\mu+1)}{q+1}.
\end{equation}
If
\begin{equation}\label{InfinityAboveCondition}
 q<(\mu+1/2)(2-\gamma),
\end{equation}
then also
\begin{equation}\label{StrictInfinity}
 M^*_{\mu,q,\gamma}>\frac1{1-\gamma}.
\end{equation}
\end{corollary}

\begin{proof}
The coefficient of \(x\) in \eqref{SmallExpansion} is positive. Hence \(R_{\mu,q,\gamma}(x)\) is larger than its zero-endpoint limit for all sufficiently small positive \(x\). If \eqref{InfinityAboveCondition} holds, the coefficient of \(x^{-1}\) in \eqref{LargeExpansion} is positive, so \(R_{\mu,q,\gamma}(x)\) is larger than its infinity-endpoint limit for all sufficiently large \(x\).
\end{proof}

\begin{proposition}\label{prop:stationary}
Let \(\mu>-1\), \(q>-1\) and \(0<\gamma<1\). If \(R'_{\mu,q,\gamma}(x_0)=0\), then
\begin{equation}\label{StationaryEq}
 R_{\mu,q,\gamma}(x_0)=
 \frac{1}{r_\mu(x_0)\{r_{\mu+1}(x_0)+(q+1)/x_0-\gamma\}}.
\end{equation}
Moreover,
\begin{equation}\label{SharpRepresentation}
 M^*_{\mu,q,\gamma}=\max\left\{
 \frac{2(\mu+1)}{q+1},
 \frac1{1-\gamma},
 \sup_{R'_{\mu,q,\gamma}(x)=0}
 \frac{1}{r_\mu(x)\{r_{\mu+1}(x)+(q+1)/x-\gamma\}}
 \right\},
\end{equation}
where the last entry is omitted if there are no stationary points.
\end{proposition}

\begin{proof}
Let
\[
 J_{\mu,q,\gamma}(x)=\exp\!\left(\gamma x\right)
 \int_0^x \exp\!\left(-\gamma t\right)t^{q-\mu}I_\mu(t)\,\dd t.
\]
Then \(J'_{\mu,q,\gamma}=\gamma J_{\mu,q,\gamma}+Y_{\mu,q}\) and \(R_{\mu,q,\gamma}=J_{\mu,q,\gamma}/Z_{\mu,q}\). Using \eqref{LogDerivatives},
\[
 R'_{\mu,q,\gamma}(x)=
 \frac1{r_\mu(x)}+
 \left(\gamma-r_{\mu+1}(x)-\frac{q+1}{x}\right)R_{\mu,q,\gamma}(x).
\]
At a stationary point this gives \eqref{StationaryEq}. It remains to justify the representation of the supremum. By Proposition \ref{prop:limits}, the quotient extends continuously to the compactified half-line with endpoint values
\(2(\mu+1)/(q+1)\) and \(1/(1-\gamma)\). Let \(B\) denote the right-hand side of \eqref{SharpRepresentation}, with the stationary-point term omitted if there are no stationary points. The inequality \(M^*_{\mu,q,\gamma}\ge B\) is immediate from the two endpoint limits and from the definition of the supremum over stationary points. Conversely, suppose that \(R_{\mu,q,\gamma}(x_0)>B\) for some \(x_0>0\). The endpoint limits allow us to choose \(0<a<x_0<b<\infty\) such that
\(R_{\mu,q,\gamma}(a)<R_{\mu,q,\gamma}(x_0)\) and \(R_{\mu,q,\gamma}(b)<R_{\mu,q,\gamma}(x_0)\). Hence \(R_{\mu,q,\gamma}\) attains its maximum on \([a,b]\) at an interior point \(y\). Then \(R'_{\mu,q,\gamma}(y)=0\), and \eqref{StationaryEq} gives \(R_{\mu,q,\gamma}(y)\le B\), contradicting \(R_{\mu,q,\gamma}(y)\ge R_{\mu,q,\gamma}(x_0)>B\). Thus no such \(x_0\) exists, and \eqref{SharpRepresentation} follows.
\end{proof}

For \(q\ge0\) the constructive constant can be optimized by a single balancing equation.

\begin{definition}\label{def:optimized}
For \(q\ge0\), define
\begin{equation}\label{Mhat}
 \widehat M_{\mu,q,\gamma}=\inf_{\gamma<\theta<1}M_{\mu,q,\gamma}(\theta),
\end{equation}
where \(M_{\mu,q,\gamma}(\theta)\) is the closed constant \eqref{Mpositive}.
\end{definition}

\begin{theorem}\label{thm:optimization}
Let \(\mu>-1\), \(q\ge0\) and \(0<\gamma<1\). There is a unique \(\theta_*=\theta_*(\mu,q,\gamma)\in(\gamma,1)\) satisfying
\begin{equation}\label{BalanceEquation}
 K_{\mu,q}\exp\!\left(\frac{2\gamma(\mu+2)\theta_*}{1-\theta_*}\right)
 =
 \frac{\mu+1+\theta_*}{(\mu+2)\theta_*(\theta_*-\gamma)}.
\end{equation}
Moreover,
\begin{equation}\label{OptimizedValue}
 \widehat M_{\mu,q,\gamma}=M_{\mu,q,\gamma}(\theta_*)
\end{equation}
and
\begin{equation}\label{SharpSandwich}
 M^*_{\mu,q,\gamma}\le \widehat M_{\mu,q,\gamma}.
\end{equation}
\end{theorem}

\begin{proof}
Write \(M_{\mu,q,\gamma}(\theta)=\max\{A(\theta),C(\theta)\}\) with
\[
 A(\theta)=K_{\mu,q}\exp\!\left(\frac{2\gamma(\mu+2)\theta}{1-\theta}\right),
 \qquad
 C(\theta)=\frac{\mu+1+\theta}{(\mu+2)\theta(\theta-\gamma)}.
\]
Then \(A\) is strictly increasing on \((\gamma,1)\) and \(C\) is strictly decreasing there, since
\[
 \frac{C'(\theta)}{C(\theta)}=\frac1{\mu+1+\theta}-\frac1\theta-\frac1{\theta-\gamma}<0.
\]
Also \(C(\theta)\to\infty\) as \(\theta\downarrow\gamma\), whereas \(A\) is finite there; and \(A(\theta)\to\infty\) as \(\theta\uparrow1\), whereas \(C\) is finite there. Hence there is a unique crossing point, and the maximum is minimized precisely at that crossing point. The bound \eqref{SharpSandwich} follows from Theorem \ref{thm:main-weighted}.
\end{proof}

\section{Parameter comparison theorems}\label{sec:parameters}

The constants have useful monotonicity properties. These are important when a family of weights or orders must be handled by a single bound.

\begin{theorem}\label{thm:constructive-monotone}
Fix \(0<\gamma<\theta<1\). On the parameter range \(\mu>-1\), \(q\ge0\), the closed constructive constant \(M_{\mu,q,\gamma}(\theta)\) in \eqref{Mpositive} is
\begin{enumerate}
\item nondecreasing in \(\mu\);
\item nonincreasing in \(q\);
\item strictly increasing in \(\gamma\).
\end{enumerate}
Consequently, the optimized constant \(\widehat M_{\mu,q,\gamma}\) is nondecreasing in \(\mu\), nonincreasing in \(q\), and nondecreasing in \(\gamma\).
\end{theorem}

\begin{proof}
For \(q\ge0\),
\[
 K_{\mu,q}=\max\left\{1,\frac{2(\mu+1)}{q+1}\right\}
\]
is nondecreasing in \(\mu\) and nonincreasing in \(q\). The exponential factor in \eqref{Mpositive} is nondecreasing in \(\mu\) and strictly increasing in \(\gamma\). The second term in \eqref{Mpositive} equals
\[
 \frac{1}{\theta(\theta-\gamma)}\frac{\mu+1+\theta}{\mu+2},
\]
which is increasing in \(\mu\) because
\[
 \frac{\dd}{\dd\mu}\frac{\mu+1+\theta}{\mu+2}=\frac{1-\theta}{(\mu+2)^2}>0,
\]
and it is strictly increasing in \(\gamma\). It is independent of \(q\). Taking a maximum preserves these monotonicities. For the optimized constant, take infima over \(\theta\); monotonicity in \(\gamma\) also uses the fact that the admissible interval \((\gamma,1)\) shrinks as \(\gamma\) increases.
\end{proof}

\begin{theorem}\label{thm:sharp-monotone}
For fixed \(\mu>-1\), the sharp constant \(M^*_{\mu,q,\gamma}\) is nonincreasing in \(q>-1\) and nondecreasing in \(\gamma\in(0,1)\).
\end{theorem}

\begin{proof}
Rewrite \eqref{Rdef} as
\begin{equation}\label{RMonotoneForm}
 R_{\mu,q,\gamma}(x)=
 \int_0^x \exp\!\left(\gamma(x-t)\right)\left(\frac tx\right)^q
 \frac{t^{-\mu}I_\mu(t)}{x^{-\mu}I_{\mu+1}(x)}\,\dd t.
\end{equation}
If \(q_2>q_1\), then \((t/x)^{q_2}\le(t/x)^{q_1}\) for \(0<t\le x\), so
\(R_{\mu,q_2,\gamma}(x)\le R_{\mu,q_1,\gamma}(x)\) for every \(x\). Taking suprema gives monotonicity in \(q\). If \(\gamma_2>\gamma_1\), then \(\exp\!\left(\gamma_2(x-t)\right)\ge\exp\!\left(\gamma_1(x-t)\right)\), so the quotient is pointwise larger and the sharp constant is nondecreasing in \(\gamma\).
\end{proof}

\begin{corollary}\label{cor:boxes}
Let \(\mu\in[\mu_0,\mu_1]\subset(-1,\infty)\), \(q\in[q_0,q_1]\subset[0,\infty)\) and \(\gamma\in(0,\gamma_0]\) with \(\gamma_0<1\). If \(\gamma_0<\theta<1\), then
\begin{equation}\label{BoxConstant}
 M_{\mu,q,\gamma}(\theta)
 \le M_{\mu_1,q_0,\gamma_0}(\theta).
\end{equation}
Thus a single explicit endpoint constant controls the whole parameter box.
\end{corollary}

\section{Positive power mixtures and monotone regular variation}\label{sec:weights}

The weighted theorem is stable under positive superposition. This permits polynomial, generalized-polynomial and power-mixture weights.

\begin{theorem}\label{thm:mixtures}
Let \(\mu>-1\), \(0<\gamma<1\), and let \(\sigma\) be a finite positive measure on a compact interval
\begin{equation}\label{Qcompact}
 Q=[q_-,q_+]\subset(-1,\infty).
\end{equation}
Define
\begin{equation}\label{PowerMixture}
 W(x)=\int_Q x^q\,\dd\sigma(q),
 \qquad x>0.
\end{equation}
For every \(\theta\in(\gamma,1)\),
\begin{equation}\label{MixtureBound}
 \int_0^x \exp\!\left(-\gamma t\right)W(t)t^{-\mu}I_\mu(t)\,\dd t
 \le M_Q(\theta)\exp\!\left(-\gamma x\right)W(x)x^{-\mu}I_{\mu+1}(x),
\end{equation}
where
\begin{equation}\label{MQdef}
 M_Q(\theta)=\sup_{q\in Q}M_{\mu,q,\gamma}(\theta)<\infty.
\end{equation}
If \(Q\subset[0,\infty)\), then by Theorem \ref{thm:constructive-monotone},
\begin{equation}\label{MQpositive}
 M_Q(\theta)=M_{\mu,q_-,\gamma}(\theta).
\end{equation}
\end{theorem}

\begin{proof}
By Tonelli's theorem and Theorem \ref{thm:main-weighted} applied to each power,
\[
\begin{aligned}
 \int_0^x \exp\!\left(-\gamma t\right)W(t)t^{-\mu}I_\mu(t)\,\dd t
 &=\int_Q\left(\int_0^x \exp\!\left(-\gamma t\right)t^{q-\mu}I_\mu(t)\,\dd t\right)\dd\sigma(q)\\
 &\le \int_Q M_{\mu,q,\gamma}(\theta)
    \exp\!\left(-\gamma x\right)x^{q-\mu}I_{\mu+1}(x)\,\dd\sigma(q)\\
 &\le M_Q(\theta)\exp\!\left(-\gamma x\right)W(x)x^{-\mu}I_{\mu+1}(x).
\end{aligned}
\]
The function \(q\mapsto M_{\mu,q,\gamma}(\theta)\) is continuous on compact subintervals of \((-1,\infty)\), so \(M_Q(\theta)<\infty\). If \(Q\subset[0,\infty)\), monotonicity in \(q\) gives \eqref{MQpositive}.
\end{proof}

\begin{corollary}\label{cor:finitesums}
Let
\begin{equation}\label{FiniteSumWeight}
 W(x)=\sum_{j=1}^N c_jx^{q_j},
 \qquad c_j>0,\qquad q_j>-1.
\end{equation}
Then
\begin{equation}\label{FiniteSumBound}
 \int_0^x \exp\!\left(-\gamma t\right)W(t)t^{-\mu}I_\mu(t)\,\dd t
 \le \left(\max_{1\le j\le N}M_{\mu,q_j,\gamma}(\theta)\right)
 \exp\!\left(-\gamma x\right)W(x)x^{-\mu}I_{\mu+1}(x).
\end{equation}
\end{corollary}

\begin{theorem}\label{thm:regularvariation}
Let \(q>-1\) and let \(L\) be positive, absolutely continuous and nondecreasing on \((0,\infty)\). Set
\begin{equation}\label{RegWeight}
 w(x)=x^qL(x).
\end{equation}
Then \(w\) is an upper \(q\)-power weight and hence satisfies \eqref{MainWeightedBound}. In particular, any positive nondecreasing slowly varying function \(L\) gives, through \(w(x)=x^qL(x)\), a regularly varying weight of index \(q\) for which the same endpoint constant \(M_{\mu,q,\gamma}(\theta)\) is valid.

More generally, if
\begin{equation}\label{RegDefect}
 \frac{L'(x)}{L(x)}\ge-\eta
 \quad\text{for a.e. }x>0,
 \qquad 0\le\eta<1-\gamma,
\end{equation}
then \(w\in\mathcal A_q(\eta)\) and the approximate endpoint bound \eqref{ApproxBound} holds.
\end{theorem}

\begin{proof}
If \(L\) is nondecreasing, then \(w(x)x^{-q}=L(x)\) is nondecreasing. This proves the first assertion. If \eqref{RegDefect} holds, then
\[
 \frac{w'(x)}{w(x)}=\frac qx+\frac{L'(x)}{L(x)}\ge\frac qx-\eta,
\]
so Definition \ref{def:approx} and Theorem \ref{thm:approx} apply.
\end{proof}

\begin{remark}\label{remark:regular}
The point of Theorem \ref{thm:regularvariation} is global rather than asymptotic. Regular variation at infinity alone does not control the lower part of the interval uniformly in \(x\). The monotonicity or exponential-defect assumption supplies exactly the missing global comparison.
\end{remark}

\section{Recovery and strengthening of Gaunt-type inequalities}\label{sec:gaunt}

We now isolate the consequences for the reciprocal-power integrals studied by Gaunt \cite{Gaunt2019}.  In particular, Corollary \ref{cor:openproblem} gives an explicit affirmative answer to Open Problem~2.10 of that paper.  Theorem \ref{thm:shifted} also removes the smallness restriction on \(\gamma\) in this family by using a single \(I_{\nu+n+1}\) endpoint with an explicit constant.

\begin{theorem}\label{thm:gaunt-family}
Let \(n>-1\), \(\nu+n>-1\), \(0<\gamma<1\), and \(\gamma<\theta<1\). Then
\begin{equation}\label{CompleteGammaGaunt}
 \int_0^x \exp\!\left(-\gamma t\right)\frac{I_{\nu+n}(t)}{t^\nu}\,\dd t
 \le M_{\nu+n,n,\gamma}(\theta)
 \exp\!\left(-\gamma x\right)\frac{I_{\nu+n+1}(x)}{x^\nu},
 \qquad x>0.
\end{equation}
If \(n\ge0\), the explicit closed constant is
\begin{equation}\label{CompleteGammaConstant}
 M_{\nu+n,n,\gamma}(\theta)=
 \max\left\{
 K_{\nu+n,n}\exp\!\left(\frac{2\gamma(\nu+n+2)\theta}{1-\theta}\right),
 \frac{\nu+n+1+\theta}{(\nu+n+2)\theta(\theta-\gamma)}
 \right\},
\end{equation}
with
\begin{equation}\label{CompleteGammaK}
 K_{\nu+n,n}=\max\left\{1,\frac{2(\nu+n+1)}{n+1}\right\}.
\end{equation}
For \(-1<n<0\), the same estimate holds with \(M_{\nu+n,n,\gamma}(\theta)\) defined by \eqref{Mgeneral} and the negative-\(q\) threshold in \eqref{Xdef}.
\end{theorem}

\begin{proof}
This is Theorem \ref{thm:shifted} with \(a=0\).
\end{proof}

\begin{corollary}\label{cor:openproblem}
Let \(\nu>-1/2\) and \(0<\gamma<1\). For every \(\theta\in(\gamma,1)\),
\begin{equation}\label{OpenProblemSolved}
 \int_0^x \exp\!\left(-\gamma t\right)\frac{I_\nu(t)}{t^\nu}\,\dd t
 <
 \max\left\{
 2(\nu+1)\exp\!\left(\frac{2\gamma(\nu+2)\theta}{1-\theta}\right),
 \frac{\nu+1+\theta}{(\nu+2)\theta(\theta-\gamma)}
 \right\}
 \exp\!\left(-\gamma x\right)\frac{I_{\nu+1}(x)}{x^\nu}
\end{equation}
for all \(x>0\). Thus \eqref{GauntOpen} holds throughout the full parameter range \(\nu>-1/2\), \(0<\gamma<1\).
\end{corollary}

\begin{proof}
Take \(n=0\) in Theorem \ref{thm:gaunt-family}. Since \(\nu>-1/2\), \(K_{\nu,0}=2(\nu+1)\), which gives the displayed constant. We record the strictness because the open problem is stated with a strict inequality. With the notation of the proof of Theorem \ref{thm:main-weighted}, put
\[
 A=2(\nu+1)\exp\!\left(\frac{2\gamma(\nu+2)\theta}{1-\theta}\right),
 \qquad
 C=\frac{\nu+1+\theta}{(\nu+2)\theta(\theta-\gamma)},
 \qquad M=\max\{A,C\}.
\]
If \(0<x\le X_{\nu,0}(\theta)\), then the proof of Theorem \ref{thm:main-weighted} uses Lemma \ref{lemma:power-unweighted} with \(q=0\). Since \(\nu>-1/2\), the coefficient ratios in \eqref{CoefficientRatio} are not all equal and the upper estimate is strict for every \(x>0\); hence the desired bound is strict.

If \(x>X_{\nu,0}(\theta)\), write again \(J=J_0+J_1\) and set \(E=\exp\{-(\theta-\gamma)(x-X_{\nu,0}(\theta))\}\). The same argument gives
\[
 J_0(x)\le AE\,x^{-\nu}I_{\nu+1}(x),
\]
whereas the tail term satisfies the strict estimate
\[
 J_1(x)< C(1-E)x^{-\nu}I_{\nu+1}(x),
\]
because Lemma \ref{lemma:ratio} gives \(r_\nu(x)>\beta_{\nu,0}(\theta)\). Therefore
\[
 J(x)<\{AE+C(1-E)\}x^{-\nu}I_{\nu+1}(x)
 \le Mx^{-\nu}I_{\nu+1}(x).
\]
Multiplying by \(e^{-\gamma x}\) proves \eqref{OpenProblemSolved}.
\end{proof}

\begin{remark}\label{remark:comparison-gaunt}
The upper estimates in \cite{Gaunt2019} for the exponentially tilted integral of \(I_{\nu+n}(t)t^{-\nu}\) were valid under a restriction of the form \(0<\gamma<1/C_{\nu,n}\). The estimate \eqref{CompleteGammaGaunt} is valid for every \(0<\gamma<1\). It does not claim to dominate the two-term endpoint expression in \cite{Gaunt2019} term by term; rather, it supplies a different endpoint majorant whose constant is explicit and uniform in \(x\) over the whole natural tilt range.
\end{remark}

\begin{corollary}\label{cor:extra-moment-gaunt}
Let \(\nu+n>-1\), \(a+n>-1\), \(0<\gamma<1\), and \(\gamma<\theta<1\). Then
\begin{equation}\label{ExtraMomentGaunt}
 \int_0^x \exp\!\left(-\gamma t\right)t^a\frac{I_{\nu+n}(t)}{t^\nu}\,\dd t
 \le M_{\nu+n,a+n,\gamma}(\theta)
 \exp\!\left(-\gamma x\right)x^a\frac{I_{\nu+n+1}(x)}{x^\nu},
 \qquad x>0.
\end{equation}
This includes \eqref{CompleteGammaGaunt} as the special case \(a=0\) and gives a complete-gamma estimate for every additional moment power \(a\) satisfying \(a+n>-1\).
\end{corollary}

\section{Conclusion}

The endpoint majorant in Gaunt's Open Problem 2.10 is the first member of a broad family of uniform weighted inequalities. The main result proves that if \(\mu>-1\), \(q>-1\), \(0<\gamma<1\), and \(w(x)x^{-q}\) is nondecreasing, then
\[
 \int_0^x \exp\!\left(-\gamma t\right)w(t)t^{-\mu}I_\mu(t)\,\dd t
 \le M_{\mu,q,\gamma}(\theta)
 \exp\!\left(-\gamma x\right)w(x)x^{-\mu}I_{\mu+1}(x),
 \qquad x>0,
\]
with \(M_{\mu,q,\gamma}(\theta)\) explicit. For \(q\ge0\) this constant is the closed expression \eqref{Mpositive}; for \(-1<q<0\) only the threshold must be enlarged as in \eqref{Xdef}.

The theory developed here adds seven pieces beyond the basic resolution of the open problem. First, the admissible-weight framework includes exact monotone power weights and approximate weights satisfying \(w'/w\ge q/x-\rho\) together with the averaged defect condition \(\int_t^x\rho(s)\,\dd s\le\eta(x-t)\), \(0\le\eta<1-\gamma\). Second, the shifted family covers \(I_{\nu+n}\) and classifies the endpoint choices \(I_{\nu+n+k}\). Third, lower endpoint bounds give lower estimates for weights with the reversed one-sided comparison, and two-sided estimates for exact power weights. Fourth, the sharp quotient is characterized by endpoint expansions and the stationary equation \eqref{StationaryEq}, with a one-parameter optimized constructive constant. Fifth, both constructive and sharp constants satisfy useful parameter comparison theorems. Sixth, positive power mixtures and monotone regularly varying amplitudes are covered without changing the endpoint scale. Seventh, the resulting estimates recover and strengthen the Gaunt-type upper bounds by covering the full natural range \(0<\gamma<1\).

The remaining problem is not existence of endpoint majorants, but sharp evaluation of \(M^*_{\mu,q,\gamma}\) beyond the stationary representation \eqref{SharpRepresentation}. That sharp-constant problem is separate from the uniform weighted upper-bound theory established here.

\acknowledgments
This research did not receive any specific grant from funding agencies in the public, commercial, or not-for-profit sectors.

\datastatement
No data were used in this study.

\section*{Declaration of Generative AI and AI-Assisted Technologies in the Writing Process}
During the preparation of this work, the authors used DeepSeek to build a specialized agent for solving mathematical problems, which was employed to generate an initial proof of the main theorem. After using this tool, the authors reviewed and edited the content as needed and take full responsibility for the content of the published article.

\paragraph*{CRediT authorship contribution statement.}
Yaoran Yang: Conceptualization, Methodology, Formal analysis, Writing - original draft. Yutong Zhang: Formal analysis, Validation, Writing - review and editing.
\paragraph*{Declaration of Generative AI and AI-Assisted Technologies in the Writing Process}
During the preparation of this work, the authors used DeepSeek for logical structuring and language polishing. After using this tool, the authors reviewed and edited the content as needed and take full responsibility for the content of the published article.
\bibliographystyle{ametsoc2014}
\bibliography{references}

\end{document}